\documentclass[reqno,10pt,a4paper]{amsart}
\usepackage[margin=2.5cm]{geometry}
\usepackage{amssymb}
\usepackage{cite}
\usepackage{lmodern}
\usepackage{mathtools}
\usepackage[usenames,dvipsnames]{color}
\usepackage{hyperref}
\hypersetup{
 colorlinks,
 linkcolor={red!50!black},
 citecolor={blue!50!black},
 urlcolor={blue!80!black}
}
\usepackage{tikz}

\usepackage{bbm}
\usepackage{mathbbol}
\newcommand{\mathbh}[1]{\mathbb{#1}}

\AtBeginDocument{%
\def\MR#1{}
}

\theoremstyle{plain}
\newtheorem{thm}{Theorem}

\newtheorem{cor}[thm]{Corollary}

\theoremstyle{definition}
\newtheorem{defn}[thm]{Definition}
\newtheorem{ex}[thm]{Example}

\theoremstyle{remark}

\title{Connectivity through bounds for the Castelnuovo-Mumford regularity}

\author{Gabriele Balletti}
\address{Department of Mathematics\\
Stockholm University, SE - 106 91 Stockholm, Sweden}
\email{balletti@math.su.se}
\date{\today}

\begin{document}


\begin{abstract}
In this note we generalize and unify two results on connectivity of
graphs: one by Balinsky and Barnette, one by Athanasiadis. This is done
through a simple proof using commutative algebra tools. In particular
we use bounds for the Castelnuovo--Mumford regularity of their
Stanley--Reisner rings. As a result, if $\Delta $ is a simplicial
$d$-pseudomanifold and $s$ is the largest integer such that
$\Delta $ has a missing face of size $s$, then the 1-skeleton of
$\Delta $ is $\left \lceil  \frac{(s+1)d}{s} \right \rceil $-connected.
We also show that this value is tight.
\end{abstract}

\maketitle

\section{Introduction}

We say that a graph $G$ having more than $m$ vertices is $m$-connected
whenever it is impossible to disconnect it by removing fewer than
$m$ vertices together with their incident edges. Interesting results on
the connectivity of $G$ can be found if $G$ is the 1-skeleton of a pure
polyhedral complex. The first step in this direction was taken in 1922
by Steinz \cite{Ste22} who characterized the 1-skeleta of
3-dimensional polytopes as the planar 3-connected graphs. Later Balinsky
\cite{Bal61} proved that the \mbox{1-skeleton} of a
$(d+1)$-dimensional convex polytope is $(d+1)$-connected. This result
has been extended to polyhedral $d$-pseudomanifolds by Barnette
\cite{Bar82}.

\begin{thm}[Balinsky, Barnette]
\label{thm:BB}
The 1-skeleton of a $d$-dimensional polyhedral pseudomanifold is
$(d+1)$-connected.
\end{thm}

In the simplicial case, Athanasiadis \cite{Ath11} proved a
stronger result for the 1-skeleton of a flag (i.e. coinciding with the
clique complex of its 1-skeleton) simplicial $d$-pseudomanifold.

\begin{thm}[Athanasiadis]
\label{thm:Ath}
The 1-skeleton of a flag simplicial d-pseudomanifold is 2d-connected.
\end{thm}

This topic has recently attracted interest among commutative
algebraists. Bj\"{o}rner and Vorwerk \cite{BV14} extended
Athanasiadis' result interpolating it with Barnette's one, thanks to a
generalization of flag complexes. Recently, Adiprasito, Goodarzi and
Varbaro \cite{AGV14} provided an algebraic method that allowed
them to obtain a more general result.

In this note we also present an algebraic approach, which allows us to
bridge the gap between Theorem~\ref{thm:BB} and Theorem~\ref{thm:Ath}
in a different direction than the one taken by \cite{BV14}
and~\cite{AGV14}. This can be done under even weaker hypotheses.

The strategy involves finding upper bounds for the Castelnuovo--Mumford
regularity of the Stanley--Reisner ring of a simplicial complex
$\Delta $ which are expressed in terms of the number of vertices of
$\Delta $. Such bounds are common in literature, as the
Castelnuovo--Mumford regularity itself acts as an upper bound for a wide
variety of algebraic invariants (such as maximum degree for which the
local cohomology modules vanish, the integer from which the Hilbert
function behaves polynomially, or the maximum degree of the syzygies).

Furthermore, we will require these bounds to behave well up to taking 
restrictions of the
simplicial complex (we call such bounds \emph{suitable}). Then we focus
on a set $T$ of vertices of $\Delta $ which have to be removed from the
1-skeleton $G$ of $\Delta $ in order to disconnect it. By setting some
hypotheses on $\Delta $, we can bound the regularity of the
Stanley--Reisner ring of the restriction $\Delta |_{T}$ from below. Such
hypotheses are weaker than requiring $\Delta $ to be a pseudomanifold.
By reversing and applying a suitable bound for the regularity, we find
a lower bound for the cardinality of the subset, therefore we can
estimate the connectivity of $G$.

\section{Preliminaries}

Let $\Delta $ be a simplicial complex on the vertex set $[n]=\{1,
\ldots ,n\}$. We call the faces of dimension 0 and 1 \emph{vertices} and
\emph{edges}, respectively. We call a face of $\Delta $ which is maximal
with respect to inclusion \emph{facet}, and we say that $\Delta $ is
\emph{pure} if its facets have the same dimension. We define the
$1$-\emph{skeleton} of a simplicial complex $\Delta $ as the set of all
the faces of $\Delta $ of dimension lower than or equal to $1$. For our
purposes we will deal with undirected simple graphs, therefore we define
a graph to be the 1-skeleton of some simplicial complex~$\Delta $. Given
a subset $T \subset [n]$ we denote by $\Delta |_{T}$ the
\emph{restriction of $\Delta $ to~$T$}, i.e. all the faces $\sigma $ of
$\Delta $ such that $\sigma \subseteq T$. A subcomplex of $\Delta $ is
called \emph{induced} if it is a restriction of $\Delta $ to some set
$T \subseteq [n]$.

Let $\mathbh{k}$ be an arbitrary field and $S=\mathbh{k}[x_{1},
\ldots ,x_{n}]$ the polynomial ring on $n$ variables. The
\emph{Stanley--Reisner ring} of the complex $\Delta $ (with respect to
the field $\mathbh{k}$) is the graded ring $\mathbh{k}[\Delta ]=S/I
_{\Delta }$ where the \emph{Stanley--Reisner ideal} $I_{\Delta }$ is the
ideal generated by all the squarefree monomials $x_{i_{1}}\cdots x
_{i_{r}} \in S$ such that $\{i_{1},\ldots ,i_{r}\} \notin \Delta $.

If $F \subseteq [n]$ and $F \notin \Delta $, but all of its proper
subsets are in $\Delta $, then we say that $F$ is a
\emph{missing face} of $\Delta $ of size $|F|$. Note that $I_{\Delta
}$ is generated by all the monomials $x_{i_{1}}\cdots x_{i_{r}}
\in S$ such that $\{i_{1},\ldots ,i_{r}\}$ is a missing face of
$\Delta $.

Let $M$ be a finitely generated graded $S$-module. Recall that the
Hilbert's Syzygy Theorem grants the existence of a \emph{minimal graded
free resolution} of $M$, i.e. a chain complex $\mathbf{F}$ of graduated
free modules of minimal rank with degree-preserving maps such that
$\mathbf{F}$ has an exact augmentation $\mathbf{F} \rightarrow M
\rightarrow 0$. Let
\[
0 \xrightarrow{}
\displaystyle
{\bigoplus_{{j \in \mathbb{N}}}} S(-j)^{\beta_{s,j}} \xrightarrow{
\phi_{s}}
\cdots \xrightarrow{\phi_{2}}
\displaystyle
{\bigoplus_{{j \in \mathbb{N}}}} S(-j)^{\beta_{1,j}} \xrightarrow{
\phi_{1}}
\displaystyle
{\bigoplus_{{j \in \mathbb{N}}}} S(-j)^{\beta_{0,j}}
\xrightarrow{}
0
\]
be a minimal graded free resolution of $M$, where the shifting numbers
$-j$ are chosen in order to let the maps $\phi_{i}$ be
degree-preserving. We call the exponents $\beta_{i,j}=\beta_{i,j}(M)$
\emph{graded Betti numbers}. Furthermore we define the
\emph{Castelnuovo--Mumford regularity} of $M$ as $\mathrm{reg}(M)=
\max \{j-i|\beta_{i,j}(M) \neq 0\}$.

We denote by $\widetilde{H}_{i}(\Delta ;\mathbh{k})$ the
$i$\textsuperscript{th} \emph{reduced (simplicial) homology} of
$\Delta $ over the field $\mathbh{k}$. Hochster's formula
\cite{Hoc77} relates the Betti numbers of the Stanley--Reisner ring
$\mathbh{k}[\Delta ]$ to the reduced homology of restrictions of
$\Delta $ as follows
\[
\beta_{i,j}(\mathbh{k}[\Delta ])=
\sum_{\substack{T \subseteq [n] \\
|T|=j}} \dim_{\mathbh{k}}\widetilde{H}_{j-i-1}(\Delta |_{T};
\mathbh{k}).
\]

It immediately follows that if some restriction $\Delta |_{T}$ of
$\Delta $ has nonzero homology in homological degree $k$ then
\begin{eqnarray}
\label{eq:ineq}
\mathrm{reg}(\mathbh{k}[\Delta ]) \geq k+1;
\end{eqnarray}
note that the equality holds whenever $k$ is the maximum integer such
that some restriction of $\Delta $ has nonzero homology in homological
degree $k$.

\section{Suitable bounds}

We now introduce a new family of simplicial complexes to which we will
extend the results on connectivity.

\begin{defn}
We say that the simplicial complex $\Delta $ is a \emph{vertex minimal
$k$-cycle} if, for some field $\mathbh{k}$, $\widetilde{H}_{k}(
\Delta |_{T};\mathbh{k})\neq 0$ if and only if $T=[n]$.
\end{defn}

Recall that a simplicial complex $\Delta $ is
\emph{strongly connected} if for every couple of facets $\tau $ and
$\sigma $ of $\Delta $, there exist a sequence $\tau_{0}, \tau_{1} ,
\ldots , \tau_{m}$ of facets of $\Delta $ with $\tau = \tau_{0}$ and
$\sigma =\tau_{m}$ such that $\tau_{i-1} \cap \tau_{i}$ is a
codimensional 1 face of both $\tau_{i-1}$ and $\tau_{i}$, for
$1 \leq i \leq m$. A \emph{simplicial $d$-pseudomanifold} is a strongly
connected simplicial complex each of whose $(d-1)$-dimensional face is
contained in exactly two facets.
As a consequence of being strongly
connected, a pseudomanifold is also pure. Note that a $d$-dimensional
simplicial pseudomanifold is a vertex minimal $d$-cycle, but is not
necessary that a vertex minimal $d$-cycle is pure or that each of its
$(d-1)$-dimensional faces is contained in exactly two facets.

In the following example we construct a $d$-dimensional vertex minimal
2-cycle for an arbitrary $d \geq 2$. For $d = 3$, it is pure, but not
strongly connected, and its \mbox{$(d-1)$}-dimensional faces are contained in
just one facet; for $d \geq 4$, $\Delta $ is not even pure.

\begin{ex}
\label{ex:cubo}
Let $Q$ be a polygon in $\mathbb{R}^{2}$ with $d+1$ vertices, for some
$d \geq 2$. Let $V$ be the set of the $2d+2$ vertices of the prism
$P=\mathrm{conv} (Q \times \{0\},Q \times \{1\}) \subset \mathbb{R}
^{3}$. We define the $d$-dimensional simplicial complex $\Delta $ on the
vertex set $V$ as the simplicial complex whose faces are all the subsets
of $V$ whose elements lie on a common face of (the boundary of)~$P$.
\end{ex}

The following theorem relates the connectivity of a vertex minimal
$k$-cycle $\Delta $ to the regularity of the Stanley--Reisner ring of
specific restrictions of $\Delta $. It allows us to prove the main
results of this note in Section~\ref{section:ultima}.

\begin{thm}
\label{thm:main}
Let $\Delta $ be a vertex minimal $k$-cycle and let $T$ be a set of
vertices of $\Delta $ such that $\Delta |_{T}$ is disconnected. Then
$\mathrm{reg}(\mathbh{k}[\Delta |_{[n]\setminus T}]) \geq k$.
\end{thm}

\begin{proof}
Let $U_{1}$ be a subset of vertices of $T$ such that $\Delta |_{U_{1}}$
is one of the connected components of $\Delta |_{T}$, and let
$U_{2} =T \setminus U_{1}$. Let $\Gamma_{1}=\Delta |_{[n]\setminus U
_{2}}$ and $\Gamma_{2}=\Delta |_{[n]\setminus U_{1}}$. Note that
$\Gamma_{1} \cup \Gamma_{2} = \Delta $ and $\Gamma_{1} \cap \Gamma
_{2} = \Delta |_{[n] \setminus T}$. By applying the Mayer--Vietoris
sequence for the reduced homology of simplicial complexes to
$\Gamma_{1}$ and $\Gamma_{2}$ we obtain the following exact sequence
\[
0 \rightarrow \widetilde{H}_{k}(\Delta ;\mathbh{k}) \rightarrow
\widetilde{H}_{k-1}(\Delta |_{[n] \setminus T };\mathbh{k}) \rightarrow
\cdots ,
\]
where the first zero comes from the hypothesis that $[n] \setminus U
_{1}$ and $[n] \setminus U_{2}$ are proper subsets of $[n]$ and
therefore $\widetilde{H}_{k}(\Gamma_{1};\mathbh{k})=\widetilde{H}
_{k}(\Gamma_{2};\mathbh{k})=0$, because $\Delta $ is a vertex minimal
$k$-cycle. Then, since $\widetilde{H}_{k}(\Delta ;\mathbh{k}) \neq 0$,
$\widetilde{H}_{k-1}(\Delta |_{[n] \setminus T };\mathbh{k})\neq 0$.

We conclude by inequality \eqref{eq:ineq}.
\end{proof}

Note that an upper bound for the regularity of the Stanley--Reisner ring
$\mathbh{k}[\Delta ]$ of a simplicial complex $\Delta $ given in terms
of the number of vertices of $\Delta $ can be reversed and applied to
Theorem~\ref{thm:main}. In this way it is possible to give a lower bound
for the number of vertices one needs to remove from $\Delta $ in order
to disconnect it, provided that the bound can be applied to restrictions
of $\Delta $.

More specifically, we will say that an upper bound for $\mathrm{reg}(
\mathbh{k}[\Delta ])$ in terms of $n$ is \emph{suitable} for the
purposes of this note, if the same bound holds true for $\mathrm{reg}(
\mathbh{k}[\Delta |_{T}])$ for each proper and nonempty subset
$T \subset [n]$ by substituting $n$ by $|T|$ in the bound itself.

A family of suitable bounds can be achieved thanks to the Taylor
resolution (\cite{Tay66}, see also \cite{BPS98}). This bound
is well known, but for the convenience of the reader we provide here a
quick argument.

Let $I$ be a monomial ideal in $S$. The degree $j$ part of the Taylor
resolution $\mathbf{F}$ of $S/I$ in homological degree $i$ must have
rank at least $\beta_{i,j}(S/I)$. If moreover $I$ is squarefree, let
$s$ be the maximum degree of one of its minimal generators $m_{1},
\ldots , m_{r}$. In this case, the multigrade vector of each generator
of the degree $j$ part of $\mathbf{F}$ in homological degree $i$ is an
element of $\{0,1\}^{n}$. Since the number of nonzero entries of this
vector cannot exceed $si$, we conclude that $\beta_{i,j}(S/I)\neq 0$
only when $j \leq si$.

We can now obtain a bound for the Castelnuovo--Mumford regularity of
$S/I$ in terms of $s$ and the number of indeterminates $n$. Indeed there
must be a nonzero Betti number $\beta_{i,j}(S/I)$ such that
$j-i=\mathrm{reg}(S/I)$. As observed before $j \leq si=s(j-
\mathrm{reg}(S/I))$, therefore, since $I$ squarefree also implies that
$j \leq n$, we obtain
\begin{eqnarray}
\label{eq:bound}
\mathrm{reg}(S/I) \leq \frac{n(s-1)}{s}.
\end{eqnarray}
This bound is suitable as the maximum degree $s$ does not increase up
to restrictions of~$\Delta $.

Improved bounds can be obtained by strengthening the hypotheses. We
report a result proved by Dao, Huneke, Schweig
\cite[Theorem~4.9]{DHS13}, where the hypotheses have been rewritten
thanks to the characterization for the $q$-step linearity given by
Eisenbud, Green, Hulek and Popescu \cite[Theorem~2.1]{EGHP05}.

\begin{thm}[\cite{DHS13}]
\label{thm:DHS}
Let $\Delta $ be a flag simplicial complex and $q$ a positive integer
such that $\Delta $ contains no induced $m$-cycles for $4 \leq m
\leq q+3$. Then
\[
\mathrm{reg}(\mathbh{k}[\Delta ]) \leq \min \left\{  \log_{
\frac{q+4}{2}}\left( \frac{n-1}{q+1} \right)  +2, \log_{\frac{q+4}{2}}
\left( \frac{(n-1) \ln (\frac{q+4}{2})}{q+1} +
\frac{2}{q+4}\right ) +2\right\} .
\]
\end{thm}

The first term of the right hand side is tighter than the second one
whenever $q\geq 2$.

Since no new induced subcycles are formed through restriction, also this
bound is suitable.

\section{Generalization of results on connectivity}\label{section:ultima}
Recall that a graph $G$ is said to be \emph{$m$-connected}, if it has
more than $m$ vertices and any subgraph obtained from $G$ by deleting
fewer than $m$ vertices and their incident edges is connected
(necessarily with at least one edge). In the language of simplicial
complexes, the previous definition can be restated as follows.

\begin{defn}
Let $G$ be 1-dimensional simplicial complex on a vertex set $V$. Then
$G$ is \emph{$m$-connected} if $|V|>m$ and $G|_{V\setminus T}$ is
connected for any subset $T \subseteq V$ with $|T| < m$.
\end{defn}

Note that by setting $|V|>m$, we exclude the trivial case of the
complete graph $K_{m}$ on $m$-vertices.

The following corollary generalizes and interpolates {Theorems~\ref{thm:BB} and \ref{thm:Ath}}.

\begin{cor}
\label{cor:BalBarAth}
Let $G$ be the 1-skeleton of a vertex minimal $k$-cycle $\Delta $, and
let $s$ be the largest integer such that $\Delta $ has a missing face
of size $s$. Then $G$ is $\left\lceil  \frac{sk}{s-1}
\right\rceil $-connected.
\end{cor}

\begin{proof}
Note that $s$ is the maximum degree of the minimal generators of the
Stanley--Reisner ideal $I_{\Delta }$. Then apply the suitable bound 
{\eqref{eq:bound}} to {Theorem~\ref{thm:main}}.
\end{proof}

Note that Balinsky--Barnette's result is obtained by looking at
$\left \lceil  \frac{sk}{s-1} \right \rceil $ for $s \gg 0$, while
Athanasiadis' one is obtained by setting $s=2$. Furthermore the class
of vertex minimal cycles is ampler than the one of minimal cycles.

We now present a family of simplicial complexes for which the previous
bound on the connectivity is tight. We thank Eran Nevo for suggesting
to build the following example in the same manner as the family of
homology spheres he introduced in \cite{Nev09}, for different
purposes. In our case, for each $s \geq 2$ and each $k \geq s-1$ we
build a vertex minimal $k$-cycle $\Delta $ whose Stanley--Reisner ideal
$I_{\Delta }$ is generated by monomials of degree not exceeding $s$ such
that it is possible to disconnect the 1-skeleton of $\Delta $ by
removing exactly $\left\lceil  \frac{sk}{s-1} \right\rceil $ vertices.
Note that the condition $k \geq s-1$ is not restrictive as $k=s-2$ holds
true only if the vertex minimal $k$-cycle $\Delta $ is the boundary of
an $(s-1)$-simplex.

\begin{ex}
\label{ex:nevo}
Let $s \geq 2$ and $k \geq s-1$ be two integers. Let $sk=(s-1)q'+r'$ the
Euclidean division of $sk$ by $s-1$, for a proper $q' \geq 0$ and a
remainder $0 \leq r' \leq s-2$. Let moreover
$\left \lceil  \frac{sk}{s-1} \right \rceil =sq+r$ be the Euclidean
division of $\left \lceil  \frac{sk}{s-1} \right \rceil $ by $s$, for a
proper $q \geq 0$ and a remainder $0 \leq r \leq s-1$. Note that
$r'=0$ if and only if $r=0$, as the second Euclidean division has no
reminder if and only if $\left \lceil  \frac{sk}{s-1} \right \rceil =
\frac{sk}{s-1}$.

We first note that $r \neq 1$. Suppose otherwise, then $r' \neq 0$ and
$\left \lceil  \frac{sk}{s-1} \right \rceil =q'+1$. The second Euclidean
division can be rewritten as
\[
q'+1=sq+1
\]
and therefore, rewriting $q'$ as $\frac{sk-r'}{s-1}$, we get
\[
sk= (s-1)sq+r'.
\]
So
\[
s(k-sq+q)=r',
\]
which is impossible, as $s\geq 2$ and $ 0 \leq r' \leq s-2$.

So the remainder $r$ must either equal $0$ or satisfy $2 \leq r
\leq s-1$. In both cases we build $\Delta $ explicitly.

Recall that the \emph{simplicial join} $\Delta_{1} * \Delta_{2}$ of two
simplicial complexes $\Delta_{1}$ and $\Delta_{2}$ on two disjoint
vertex sets is the simplicial complex whose faces can be written as the
union of a face of $\Delta_{1}$ with a face of $\Delta_{2}$. Moreover,
recall that the simplicial join of two spheres of dimension
$d_{1}$ and $d_{2}$ is a $(d_{1} + d_{2} +1)$-sphere.

If $r=r'=0$, we define $\Delta =\partial \sigma^{1} * \partial
\sigma^{s-1} * \cdots *\partial \sigma^{s-1}$, where $\partial \sigma
^{i}$ denotes the boundary of an $i$-dimensional simplex, and
$\partial \sigma^{s-1}$ appears $q$ times in the join. In this way,
$\Delta $~is a sphere of dimension $q(s-1)$. We now prove that
$q(s-1)=k$; by definition of $q$ we get
\[
q(s-1)=\frac{sk}{s(s-1)} (s-1)= k.
\]

Conversely, let $2 \leq r \leq s-1$. As observed before, $r'\neq 0$. In
this case we define $\Delta $ as the join $\partial \sigma^{1} *
\partial \sigma^{s-1} * \cdots *\partial \sigma^{s-1} * \partial
\sigma^{r-1}$, where $\partial \sigma^{s-1}$ appears $q$ times. In this
way, $\Delta $~is a $(q(s-1)+r-1)$-sphere. We now prove that
$q(s-1)+r-1=k$. Indeed,
\[
q(s-1)+r-1 = \frac{q'+1-r}{s}(s-1)+r -1=k -\frac{r-r'-1}{s}.
\]
The quantity $\frac{r-r'-1}{s}$ has to be an integer, and since
$0 \leq r\leq s-1$ and $0 \leq r' \leq s-2$, the only integer value it
can equal is 0.

In both the cases $\Delta $ is an $k$-sphere on
$\left\lceil  \frac{sk}{s-1} \right\rceil +2$ vertices, indeed the second
Euclidean division counts exactly the number of vertices of the join
except for the two vertices of $\partial \sigma^{1}$. Moreover, in both
the cases, the largest $i$ such that $\partial \sigma^{i}$ is in
$\Delta $ is $s-1$, and therefore the Stanley--Reisner ideal
$I_{\Delta }$ is generated by monomials of degree at most $s$. So
$\Delta $ satisfies the hypothesis of Corollary~\ref{cor:BalBarAth}.

If we remove from the 1-skeleton of $\Delta $ all the vertices but the
two belonging to $\partial \sigma^{1}$, we disconnect it. Note that we
are removing $n-2=\left\lceil  \frac{sk}{s-1} \right\rceil $ vertices,
therefore the 1-skeleton of $\Delta $ can not be more than
$\left\lceil  \frac{sk}{s-1} \right\rceil $-connected.
\end{ex}

If we have sufficient hypotheses to apply the bound for the regularity
given by Dao, Huneke, Schweig (see Theorem~\ref{thm:DHS}) we can obtain
the following result in which the connectivity of the simplicial complex
grows exponentially on $k$.

\begin{cor}
Let $G$ be the 1-skeleton of a vertex minimal $k$-cycle $\Delta $ which
is flag and without induced $q$-cycles for $q \geq 4$. Then $G$ is
$M$-connected, where
\[
M=\max \left \{  \left\lceil  (q+1)\left( \frac{q+4}{2} \right) ^{k-2} +1
\right\rceil ,
\left\lceil  \frac{q+1}{\ln (\frac{q+4}{2})}\left( \left( \frac{q+4}{2} \right)
^{k-2}-\frac{2}{q+4}\right) +1 \right \rceil  \right\} .
\]
\end{cor}

\begin{proof}
Apply Theorem~\ref{thm:DHS} to Theorem~\ref{thm:main}.
\end{proof}

For the sake of readability, we emphasize that from the corollary it
follows that $G$ results at least $\left\lceil  \left( \frac{q}{2}\right)
^{k-1}\right\rceil $-connected.

A family of simplicial pseudomanifolds of arbitrary dimension which
satisfy the hypotheses of the previous corollary has been built by
Januszkiewicz and \'{S}wi{\c{a}}tkowski in \cite{JS03}.

\subsection*{Acknowledgements}
The author is extremely grateful to Matteo Varbaro for his valuable
support and advice. He also thanks Afshin Goodarzi and J\"{o}rgen
Backelin for the careful proofreading, Benjamin Nill and Alessio
D'Al\`{i} for the useful suggestions and corrections, Eran Nevo for
leading the author to {Example~\ref{ex:nevo}}, and the anonymous referees
for the insightful and patient comments.

\bibliographystyle{plain}

\begin{thebibliography}{10}

\bibitem{AGV14}
Karim~A. Adiprasito, Afshin Goodarzi, and Matteo Varbaro.
\newblock Connectivity of pseudomanifold graphs from an algebraic point of
  view.
\newblock {\em C. R. Math. Acad. Sci. Paris}, 353(12):1061--1065, 2015.

\bibitem{Ath11}
Christos~A. Athanasiadis.
\newblock Some combinatorial properties of flag simplicial pseudomanifolds and
  spheres.
\newblock {\em Ark. Mat.}, 49(1):17--29, 2011.

\bibitem{Bal61}
Michel~Louis Balinski.
\newblock On the graph structure of convex polyhedra in $n$-space.
\newblock {\em Pacific Journal of Mathematics}, 11(2):431--434, 1961.

\bibitem{Bar82}
David Barnette.
\newblock Decompositions of homology manifolds and their graphs.
\newblock {\em Israel Journal of Mathematics}, 41(3):203--212, 1982.

\bibitem{BPS98}
Dave Bayer, Irena Peeva, and Bernd Sturmfels.
\newblock Monomial resolutions.
\newblock {\em Math. Res. Lett.}, 5(1-2):31--46, 1998.

\bibitem{BV14}
Anders Bj{\"o}rner and Kathrin Vorwerk.
\newblock On the connectivity of manifold graphs.
\newblock {\em Proc. Amer. Math. Soc.}, 143(10):4123--4132, 2015.

\bibitem{DHS13}
Hailong Dao, Craig Huneke, and Jay Schweig.
\newblock Bounds on the regularity and projective dimension of ideals
  associated to graphs.
\newblock {\em J. Algebraic Combin.}, 38(1):37--55, 2013.

\bibitem{EGHP05}
David Eisenbud, Mark Green, Klaus Hulek, and Sorin Popescu.
\newblock Restricting linear syzygies: algebra and geometry.
\newblock {\em Compos. Math.}, 141(6):1460--1478, 2005.

\bibitem{Hoc77}
Melvin Hochster.
\newblock {Cohen-{M}acaulay rings, combinatorics, and simplicial complexes}.
\newblock In {\em Ring theory, {II} ({P}roc. {S}econd {C}onf., {U}niv.
  {O}klahoma, {N}orman, {O}kla., 1975)}, pages 171--223. Lecture Notes in Pure
  and Appl. Math., V. Dekker, New York, 1977.

\bibitem{JS03}
Tadeusz Januszkiewicz and Jacek {\'S}wi{\c{a}}tkowski.
\newblock Hyperbolic {C}oxeter groups of large dimension.
\newblock {\em Comment. Math. Helv.}, 78(3):555--583, 2003.

\bibitem{Nev09}
Eran Nevo.
\newblock {Remarks on missing faces and generalized lower bounds on face
  numbers}.
\newblock {\em Electron. J. Combin.}, 16(2, Special volume in honor of Anders
  Bjorner):Research Paper 8, 11, 2009.

\bibitem{Ste22}
Ernst Steinitz.
\newblock {\em Polyeder und Raumeinteilungen}.
\newblock Encyclop{\"a}die der mathematischen Wissenschaften, Band 3
  (Geometries), 1922.

\bibitem{Tay66}
Diana~Kahn Taylor.
\newblock {\em Ideals {G}enerated by {M}onomials in an {R}-{S}equence}.
\newblock ProQuest LLC, Ann Arbor, MI, 1966.
\newblock Thesis (Ph.D.)--The University of Chicago.

\end{thebibliography}

\end{document}